\documentclass[letterpaper, 11pt,  reqno]{amsart}

\usepackage{amsmath,amssymb,amscd,amsthm,amsxtra, esint}
\usepackage{tikz} 
\usepackage{color}
\usepackage[implicit=true]{hyperref}

\usepackage{cases}

\usepackage[left=32mm, right=32mm, 
bottom=27mm]{geometry}

\allowdisplaybreaks[2]

\usepackage{color}

\definecolor{gr}{rgb}   {0.,   0.69,   0.23 }
\definecolor{bl}{rgb}   {0.,   0.5,   1. }
\definecolor{mg}{rgb}   {0.85,  0.,    0.85}
\definecolor{yl}{rgb}   {0.8,  0.7,   0.}
\definecolor{or}{rgb}  {0.7,0.2,0.2}

\usetikzlibrary{shapes.misc}
\usetikzlibrary{shapes.symbols}
\usetikzlibrary{shapes.geometric}

\tikzset{
	ddot/.style={circle,fill=white,draw=black,inner sep=0pt,minimum size=0.8mm},
	>=stealth,
	}

\tikzset{
	ddot2/.style={circle,fill=black,draw=black,inner sep=0pt,minimum size=0.8mm},
	>=stealth,
	}

\newtheorem{theorem}{Theorem} [section]

\newtheorem{lemma}[theorem]{Lemma}
\newtheorem{proposition}[theorem]{Proposition}
\newtheorem{remark}[theorem]{Remark}

\newtheorem*{acknowledgment}{Acknowledgments}

\DeclareMathOperator*{\supp}{supp}

\newcommand{\I}{\mathcal{I}}

\newcommand{\noi}{\noindent}
\newcommand{\Z}{\mathbb{Z}}
\newcommand{\R}{\mathbb{R}}

\newcommand{\T}{\mathbb{T}}

\newcommand{\D}{\mathcal{D}}

\let\Re=\undefined\DeclareMathOperator*{\Re}{Re}
\let\Im=\undefined\DeclareMathOperator*{\Im}{Im}

\let\P= \undefined
\newcommand{\P}{\mathbf{P}}

\newcommand{\Q}{\mathbf{Q}}

\newcommand{\E}{\mathbb{E}}

\renewcommand{\L}{\mathcal{L}}

\newcommand{\al}{\alpha}
\newcommand{\be}{\beta}
\newcommand{\dl}{\delta}

\newcommand{\nb}{\nabla}

\newcommand{\deff}{\stackrel{\textup{def}}{=}}

\newcommand{\Dl}{\Delta}
\newcommand{\eps}{\varepsilon}
\newcommand{\kk}{\kappa}
\newcommand{\g}{\gamma}
\newcommand{\G}{\Gamma}

\newcommand{\s}{\sigma}

\newcommand{\ft}{\widehat}

\newcommand{\cj}{\overline}

\newcommand{\dt}{\partial_t}

\renewcommand{\o}{\omega}
\renewcommand{\O}{\Omega}

\newcommand{\les}{\lesssim}
\newcommand{\ges}{\gtrsim}

\newcommand{\Id}{\mathrm{Id}}

\newcommand{\jb}[1]
{\langle #1 \rangle}

\newcommand{\ind}{\mathbf 1}

\newcommand{\N}{\mathbb{N}}

\renewcommand{\H}{\mathcal{H}}

\newcommand{\U}{\Theta}

\renewcommand{\Q}{\mathbb{Q}}
\newcommand{\PP}{\mathbb{P}}

\newcommand{\ZZ}{\mathcal{Z}}

\newcommand{\muu}{\vec{\mu}}

\newcommand{\rhoo}{\vec{\rho}}

\newcommand{\W}{\mathcal{W}}

\newcommand{\dr}{\eta}
\newcommand{\Dr}{H}

\newcommand{\Ha}{\mathbb{H}_a}
\newcommand{\Hc}{\mathbb{H}_c}

\numberwithin{equation}{section}
\numberwithin{theorem}{section}

\begin{document}
\baselineskip = 15pt

\title[Invariant Gibbs measure for the sine-Gordon model]
{Invariant Gibbs dynamics for the dynamical sine-Gordon model}

\author[T.~Oh, T.~Robert, P.~Sosoe, and Y.~Wang]
{Tadahiro Oh, Tristan Robert, Philippe Sosoe, and Yuzhao Wang}

\address{
Tadahiro Oh, School of Mathematics\\
The University of Edinburgh\\
and The Maxwell Institute for the Mathematical Sciences\\
James Clerk Maxwell Building\\
The King's Buildings\\
Peter Guthrie Tait Road\\
Edinburgh\\ 
EH9 3FD\\
 United Kingdom}

\email{hiro.oh@ed.ac.uk}

\address{
Tristan Robert\\
Fakult\"at f\"ur Mathematik\\
Universit\"at Bielefeld\\
Postfach 10 01 31\\
33501 Bielefeld\\
Germany}

\email{trobert@math.uni-bielefeld.de}

\address{
Philippe Sosoe\\
 Department of Mathematics\\ Cornell University, 584 Malott Hall, Ithaca\\
New York 14853, USA}
\email{psosoe@math.cornell.edu}

\address{
Yuzhao Wang\\
School of Mathematics, 
University of Birmingham, 
Watson Building, 
Edgbaston, 
Birmingham\\
B15 2TT, 
United Kingdom}

\email{y.wang.14@bham.ac.uk}



\begin{abstract}
In this note, 
we  study the hyperbolic stochastic damped sine-Gordon equation (SdSG),
with a parameter $\be^2 > 0$,  
 and its associated Gibbs dynamics
on the two-dimensional torus.
After introducing a suitable renormalization, 
we first  construct  the Gibbs measure in the range $0<\be^2<4\pi$
via the variational approach due to Barashkov-Gubinelli (2018).
We then prove 
 almost sure global well-posedness and invariance of the Gibbs measure under 
 the hyperbolic SdSG dynamics in the range $0<\be^2<2\pi$.
Our construction of the Gibbs measure also yields
almost sure global well-posedness and invariance of the Gibbs measure
for the parabolic sine-Gordon model 
in the range $0<\be^2<4\pi$.
\end{abstract}


\maketitle
%


\baselineskip = 13.5pt

\vspace*{-4mm}

\section{Introduction}
\subsection{Dynamical sine-Gordon model}

We consider
the following stochastic damped sine-Gordon equation (SdSG) on 
$\T^2 = (\R/2\pi\Z)^2$
with an additive space-time white noise forcing:
\begin{align}
\begin{cases}
\dt^2 u + \dt u + (1- \Dl)  u   +  \g \sin(\be u) = \sqrt{2}\xi\\
(u, \dt u) |_{t = 0} = (u_0, v_0) , 
\end{cases}
\qquad (t, x) \in \R_+\times\T^2,
\label{SdSG}
\end{align}

\noi
where 
$\g$ and $\be$ are non-zero real numbers
and 
$\xi$ denotes a (Gaussian) space-time white noise on $\R_+\times\T^2$.
Our main goal in this paper is 
to construct 
invariant dynamics of SdSG~\eqref{SdSG}
associated with the 
 Gibbs measure, which formally reads
\begin{align}\label{Gibbs1}
``d\rhoo(u,v) = \ZZ^{-1}e^{-E(u,v)}dudv".
\end{align}

\noi
Here,  $\ZZ = \ZZ(\be)$
denotes a normalization constant
 and 
 \begin{align}
E(u,v)= \frac12\int_{\T^2}\big(u(x)^2+|\nabla u(x)|^2 + v(x)^2 \big)dx -\frac\g\be \int_{\T^2}\cos\big(\be u(x)\big)dx
\label{Hamil1}
\end{align}

\noi
 denotes the energy (= Hamiltonian) of the (deterministic undamped) sine-Gordon equation:
\begin{align}
\dt^2 u +  (1- \Dl)  u   +  \g \sin(\be u) = 0.
\label{SdSG2}
\end{align}

\noi
Our first goal is to provide a rigorous construction of 
 the Gibbs measure $\rhoo$ for $0 < \be^2 < 4\pi$;
 see Theorem \ref{THM:Gibbs}.

The Gibbs measure $\rhoo$ in \eqref{Gibbs1} 
arises in various  physical contexts such as two-dimensional Yukawa and Coulomb gases in statistical mechanics and  the quantum sine-Gordon model in Euclidean quantum field theory. 
We refer the readers to \cite{PS, BEMS, Fro,McKean81,McKean94,LRV, LRV2,HS,CHS} and the references therein for more physical motivations and interpretations of the measure~$\rhoo$. 
The dynamical model \eqref{SdSG} then corresponds to the 
so-called ``canonical" stochastic quantization~\cite{RSS} of the quantum sine-Gordon model represented by the measure $\rhoo$ in \eqref{Gibbs1}.

From the analytical point of view, 
the hyperbolic SdSG  \eqref{SdSG} is  a good  model 
for the  study of singular stochastic nonlinear wave equations (SNLW). 
SNLW has been studied in various settings; 
see for example \cite[Chapter 13]{DPZ14} and the references therein. 
In particular,  over the past several years, 
we have witnessed a fast development
in the Cauchy theory of
singular SNLW on $\T^d$.
When $d = 2$, 
the well-posedness theory for 
SNLW with a 
 polynomial nonlinearity:
\begin{align}\label{SNLW}
\dt^2u+(1-\Dl)u+u^k=\xi
\end{align}

\noi
 is now well understood \cite{GKO,GKOT, OOR, OO}. See also \cite{ORTz, To}
 for related results on two-dimensional compact Riemannian manifolds
 \cite{ORTz}
 and on $\R^2$ \cite{To}.
  The situation is more delicate for $d = 3$.
In a recent paper, \cite{GKO2},  
Gubinelli, Koch, and the first author treated the quadratic case ($k = 2$)
by adapting the paracontrolled calculus, originally introduced in the parabolic setting \cite{GIP},
 to the dispersive setting.
For the sine-Gordon model, the main new difficulty in  \eqref{SdSG} 
comes from the non-polynomial nature of the nonlinearity,
which makes the analysis of the relevant stochastic object particularly non-trivial.

In  the aforementioned works,  the main source of difficulty 
 comes from the  roughness of the space-time white noise $\xi$.
For $d \geq 2$,  
 the stochastic convolution $\Psi$, 
 solving the following 
  linear stochastic wave equation:\footnote{The equation \eqref{SLW}  is also referred to as the linear stochastic  Klein-Gordon equation.  In the following, however, we  simply refer to this as to 
  the  wave equation.}
\begin{align}\label{SLW}
\dt^2\Psi + (1-\Dl)\Psi =\xi, 
\end{align}

\noi
belongs 
almost surely to $C(\R_+;W^{-\eps,\infty}(\T^2))$ for any $\eps>0$
but not for $\eps = 0$.
See Lemma~\ref{LEM:psi}.
This lack of regularity shows that 
there is an issue in forming a
nonlinearity of the form  $\Psi^k$
and $\sin (\be \Psi)$, 
thus requiring a proper renormalization.
In our previous work~\cite{ORSW}, 
we studied the undamped case:
\begin{align}\label{SSG}
\dt^2 u + (1-\Dl)u + \g\sin(\be u)=\xi.
\end{align} 

\noi
 By introducing a {\it time-dependent} renormalization, 
we proved local well-posedness of 
the undamped model 
\eqref{SSG} for {\it any} value of $\be^2>0$ for small times (depending on $\be$).
The main ingredient in \cite{ORSW} was
 to exploit  smallness of $\Psi$ in \eqref{SLW} for small times
(thanks to the time-dependent nature of the renormalization).

The situation for the damped model \eqref{SdSG} is, however, 
different from the undamped case.
In studying the problem associated with 
the  (formally) invariant measure $\rhoo$ in \eqref{Gibbs1}, 
we work with a {\it time-independent} renormalization
and thus the situation is closer to the parabolic model:\footnote{We point out that 
while
the spatially homogeneous case with $\dt -\frac12 \Dl$ was studied in 
 \cite{HS, CHS}, 
the model~\eqref{PSG} is more relevant for our discussion.
See Remark \ref{REM:PSG}.}
\begin{align}\label{PSG}
\dt u +\tfrac12(1-\Dl) u + \g\sin(\be u)=\xi,
\end{align}

\noi
studied in \cite{HS, CHS}, 
where the value of $\be^2 > 0$ played an important role in the solution theory.
 For $0<\be^2<4\pi$,  the Da Prato-Debussche trick \cite{DPD}
along with a standard Wick renormalization yields
local well-posedness of \eqref{PSG}; see Remark~\ref{REM:PSG}.
It turns out that there is an infinite number of thresholds:
$\be^2 = \frac{j}{j+1}8\pi$, $j\in\N$, 
where one encounters new divergent stochastic objects, requiring further renormalizations.
By using the theory of regularity structures~\cite{H}, 
Hairer-Shen \cite{HS} and Chandra-Hairer-Shen~\cite{CHS}
proved  local well-posedness of  the parabolic model \eqref{PSG}  to the entire subcritical regime 
$0 < \be^2<8\pi$. 
When $\be^2 = 8 \pi$, 
the equation~\eqref{PSG}  is critical
and falls outside the scope of the current theory.

Due to a weaker smoothing property of
the relevant linear propagator, 
our hyperbolic model \eqref{SdSG}
is expected to be much more involved than the parabolic case.
Indeed, as we see below, 
the standard Da Prato-Debussche trick
 yields the solution theory for \eqref{SdSG}
 only for $0 < \be^2 < 2\pi$, 
(which is much smaller than the parabolic case: $0 < \be^2 < 4\pi$).
See Theorem \ref{THM:GWP} below.
 
In the next subsection, we provide precise statements of our main results.
Before proceeding further, 
we mention 
the recent works \cite{Ga,HKK,ORW} 
on the well-posedness theory for the stochastic heat and wave equations with an exponential nonlinearity
in the two-dimensional setting.

\subsection{Main results}
Our main goal in this paper is two-folds;
(i) provide a rigorous construction of (a renormalized version of) the Gibbs measure $\rhoo$ in \eqref{Gibbs1}
and (ii) construct well-defined dynamics for the hyperbolic SdSG \eqref{SdSG}
associated with the Gibbs initial data.
For this purpose,  we first fix some notations. 
Given $ s \in \R$, 
let $\mu_s$ denote
a Gaussian measure,   formally defined by
\begin{align}
 d \mu_s 
   = Z_s^{-1} e^{-\frac 12 \| u\|_{{H}^{s}}^2} du
& =  Z_s^{-1} \prod_{n \in \Z^2} 
 e^{-\frac 12 \jb{n}^{2s} |\ft u_n|^2}   
 d\ft u_n , 
\label{gauss0}
\end{align}

\noi
where 
  $\jb{\,\cdot\,} = \big(1+|\,\cdot\,|^2\big)^\frac{1}{2}$
and $\ft u_n$  denotes the Fourier transforms of $u$.
We set 
\begin{align}
\muu_s = \mu_s \otimes \mu_{s-1} .
\label{gauss1}
\end{align}

\noi
In particular, when $s = 1$, 
 the measure $\muu_1$ is defined as 
   the induced probability measure
under the map:
\begin{equation*}
\o \in \O \longmapsto (u^\o, v^\o),
 \end{equation*}

\noi
where $u^\o$ and $v^\o$ are given by
\begin{equation}\label{series}
u^\o = \sum_{n \in \Z^2} \frac{g_n(\o)}{\jb{n}}e_n
\qquad\text{and}\qquad
v^\o = \sum_{n \in \Z^2} h_n(\o)e_n.
\end{equation}

\noi
Here, 
 $e_n=(2\pi)^{-1}e^{i n\cdot x}$
 and $\{g_n,h_n\}_{n\in\Z^2}$ denotes  a family of independent standard 
 complex-valued  Gaussian random variables conditioned so that $\cj{g_n}=g_{-n}$ and $\cj{h_n}=h_{-n}$, 
 $n \in \Z^2$.
It is easy to see that $\muu_1 = \mu_1\otimes\mu_0$ is supported on
\begin{align*}
\H^{s}(\T^2)\deff H^{s}(\T^2)\times H^{s - 1}(\T^2)
\end{align*}

\noi
for $s < 0$ but not for $s \geq 0$.

With  
 \eqref{Hamil1}, 
 \eqref{gauss0}, and 
 \eqref{gauss1},  
 we can formally write $\rhoo$ in \eqref{Gibbs1} as 
\begin{align}\label{Gibbs2}
d\rhoo(u,v) \sim e^{\frac{\g}{\be}\int_{\T^2}\cos(\be u)dx}
d\muu_1 (u, v).
\end{align} 

\noi
In view of the roughness of the support of $\muu_1$, 
the nonlinear term in \eqref{Gibbs2} is not well-defined
and thus a proper renormalization is required to give a meaning to \eqref{Gibbs2}.

Let  $\P_N$  be a smooth frequency projector
onto the frequencies  $\{n\in\Z^2:|n|\leq N\}$
defined as a Fourier multiplier operator with a symbol:
\begin{align}
\chi_N(n) = \chi(N^{-1}n)
\label{chi}
\end{align}

\noi
for some  fixed non-negative function 
 $\chi \in C^\infty_0(\R^2)$ 
such that $\supp \chi \subset \{\xi\in\R^2:|\xi|\leq 1\}$ and $\chi\equiv 1$ 
on $\{\xi\in\R^2:|\xi|\leq \tfrac12\}$.  
Given  $u = u^\o$ as in \eqref{series}, 
i.e.\footnote{Given a random variable $X$, $\L(X)$ denotes the law of $X$.}~$\L(u) = \mu_1$, 
set $\s_N$, $N \in \N$, by setting
 \begin{align}\label{sN}
 \s_N =  \E\Big[\big(\P_Nu(x)\big)^2\Big] =\frac1{4\pi^2}\sum_{n\in\Z^2}\frac{\chi_N(n)^2}{\jb{n}^2}
 = \frac1{2\pi}\log N + o(1),
 \end{align}

\noi
as $N \to \infty$, 
independent of $x\in\T^2$.
Given $N \in \N$, 
define 
 the truncated renormalized density:
\begin{align}\label{RN}
R_N(u) = \frac{\g_N}{\be}\int_{\T^2}\cos\big(\be\P_Nu(x)\big)dx,
\end{align}

\noi
where $\g_N = \g_N(\be)$ is defined by 
 \begin{equation}\label{CN}
 \g_N(\be) = e^{\frac{\be^2}{2}\s_N}. 
 \end{equation}

\noi
In particular, we have  $\g_N \to \infty$ as $N \to \infty$.
We then  define the truncated renormalized Gibbs measure:
\begin{align}\label{GibbsN}
d\rhoo_N(u,v)= \ZZ_N^{-1}e^{R_N(u)}d\muu_1(u,v)
\end{align}

\noi
for some normalization constant $\ZZ_N = \ZZ_N(\be) \in (0,\infty)$.
We now state our first result.

\begin{theorem}\label{THM:Gibbs}
Let $0<\be^2<4\pi$. 

\noi
\textup{(i)} The truncated renormalized density $\{R_N\}_{N\in\N}$ in \eqref{RN} 
is a Cauchy sequence in $L^p(\mu_1)$ for any finite $p\ge 1$, thus converging to some limiting random variable $R\in L^p(\mu_1)$.

\smallskip
\noi
\textup{(ii)} 
Given any finite $ p \ge 1$, 
there exists $C_p > 0$ such that 
\begin{equation}
\sup_{N\in \N} \Big\| e^{R_N(u)}\Big\|_{L^p(\mu_1)}
\leq C_p  < \infty.
\label{exp1}
\end{equation}

\noi
Moreover, we have
\begin{equation}\label{exp2}
\lim_{N\rightarrow\infty}e^{ R_N(u)}=e^{R(u)}
\qquad \text{in } L^p(\mu_1).
\end{equation}

\noi
As a consequence, 
the truncated renormalized Gibbs measure $\rhoo_N$ in \eqref{GibbsN} converges, in the sense of \eqref{exp2}, 
 to the renormalized Gibbs measure $\rhoo$ given by
\begin{align}\label{Gibbs3}
d\rhoo(u,v)= \ZZ^{-1} e^{R(u)}d\muu_1(u, v).
\end{align}

\noi
Furthermore, 
the resulting Gibbs measure $\rhoo$ is equivalent 
to the Gaussian measure $\muu_1$.

\end{theorem}

The proof of Theorem \ref{THM:Gibbs} also allows us to define the renormalized Gibbs measure:
\begin{align}
d\rho(u)= \ZZ^{-1} e^{R(u)}d\mu_1(u)
\label{Gibbs4}
\end{align}

\noi
as a limit of the truncated measure 
\begin{align*}
d\rho_N(u) =  \ZZ_N^{-1}e^{R_N(u)}d\mu_1(u)
\end{align*}

\noi
for $0 < \be^2 < 4\pi$.
The Gibbs measure $\rho$ in \eqref{Gibbs4}
is relevant to the parabolic model \eqref{PSG}.
See Remark~\ref{REM:PSG}.

In a recent work \cite{LRV2}, Lacoin, Rhodes, and Vargas 
constructed a measure associated with the sine-Gordon model
in the one-dimensional setting, where the based Gaussian measure is log-correlated 
(as in the massive Gaussian free field on $\T^2$).
Their construction applies to the 
full subcritical range:\footnote{Due to a different scaling, the threshold $0<\be^2<2d$ in \cite{LRV2} corresponds to $0<\be^2<8\pi$ in our convention. See Remark 1.14 in \cite{ORW}.} 
 $0<\be^2<8\pi$.
At this moment, their argument is restricted to the one-dimensional case and 
does not extend to the two-dimensional case under consideration.

Theorem \ref{THM:Gibbs}\,(i) 
follows from  the construction of the imaginary Gaussian multiplicative chaos; see Lemma \ref{LEM:Ups} below. 
The main difficulty in proving Theorem \ref{THM:Gibbs}\,(ii)
appears in showing the uniform bound \eqref{exp1}.
We  establish the bound \eqref{exp1}
by applying the variational approach
 introduced by Barashkov and Gubinelli in \cite{BG}
  in the construction of  the $\Phi^4_3$-measure. See also \cite{GOTW}.

\medskip

Next, 
we move onto the well-posedness theory 
of the hyperbolic SdSG \eqref{SdSG}.
Let us first introduce  the following  renormalized truncated SdSG:
\begin{align}
\dt^2 u_N   + \dt u_N  +(1-\Dl)  u_N 
+\g_N\P_N\big\{\sin (\be \P_Nu_N)\big\}   = \sqrt{2} \xi , 
\label{RSdSGN}
\end{align} 

\noi
where 
 $\g_N$ is as in  \eqref{CN}. 
We now state our second result.

 \begin{theorem}\label{THM:GWP}
Let  $0<\be^2<2\pi$. Then,
the stochastic damped sine-Gordon equation~\eqref{SdSG} is almost surely globally well-posed with respect to the renormalized Gibbs measure~$\rhoo$ in~\eqref{Gibbs3}. Furthermore, the renormalized Gibbs measure $\rhoo$ is invariant under the dynamics.

More precisely, there exists a non-trivial stochastic process $(u,\dt u)\in C(\R_+;\H^{-\eps}(\T^2))$ for any $\eps>0$ such that, for any $T>0$, the solution $(u_N,\dt u_N)$ to 
the truncated SdSG~\eqref{RSdSGN} with 
the random initial data 
$(u_N, \dt u _N)|_{t = 0}$ distributed according 
to the truncated Gibbs measure $\rhoo_N$  in~\eqref{GibbsN}, converges in probability to $(u,\dt u)$ in $C([0,T];\H^{-\eps}(\T^2))$.
Moreover,  the law of $(u(t),\dt u(t))$ is given by the renormalized Gibbs measure $\rhoo$ in \eqref{Gibbs3}
 for any $t\ge 0$.
\end{theorem}

In view of Theorem~\ref{THM:Gibbs} and Bourgain's invariant measure argument \cite{BO94,BO96},
Theorem \ref{THM:GWP}
follows once we construct the limiting process $(u,\dt u)$ locally in time. 
Furthermore, in view of the equivalence of 
 $\rhoo$ and $\muu_1$, 
 it suffices to study the dynamics with the Gaussian random initial data $(u_0, v_0)$ 
 with $\L(u_0, v_0) = \muu_1$.
 
 As in \cite{ORSW}, we proceed with the Da Prato-Debussche trick.
  For our damped model, 
we  let $\Psi$  be the solution to the linear stochastic damped wave equation:
\begin{align}\label{SdLW}
\begin{cases}
\dt^2 \Psi + \dt\Psi +(1-\Dl)\Psi  = \sqrt{2}\xi\\
(\Psi,\dt\Psi)|_{t=0}=(u_0,v_0),
\end{cases}
\end{align}

\noi
where $\L (u_0,v_0) = \muu_1$. 
Define the linear damped wave propagator $\D(t)$ by 
\begin{equation}\label{D}
\D(t) = e^{-\frac{t}2}\frac{\sin\Big(t\sqrt{\frac34-\Dl}\Big)}{\sqrt{\frac34-\Dl}}
\end{equation} 

\noi
as a Fourier multiplier operator.
Then, we have 
\begin{align} 
\Psi (t) 
 = \dt\D(t)u_0 + \D(t)(u_0+v_0)+ \sqrt{2}\int_0^t\D(t - t')dW(t'), 
\label{PsiN}
\end{align}

\noi
where  $W$ denotes a cylindrical Wiener process on $L^2(\T^2)$:
\begin{align}
W(t) =  \sum_{n \in \Z^2 } B_n (t) e_n,
\label{W1}
\end{align}

\noi
and  
$\{ B_n \}_{n \in \Z^2}$ 
is defined by 
$B_n(0) = 0$ and 
$B_n(t) = \jb{\xi, \ind_{[0, t]} \cdot e_n}_{ t, x}$.
Here, $\jb{\cdot, \cdot}_{t, x}$ denotes 
the duality pairing on $\R \times \T^2$.
As a result, 
we see that $\{ B_n \}_{n \in \Z^2}$ is a family of mutually independent complex-valued\footnote
{In particular, $B_0$ is  a standard real-valued Brownian motion.} 
Brownian motions conditioned so that $B_{-n} = \cj{B_n}$, $n \in \Z^2$. 
By convention, we normalized $B_n$ such that $\text{Var}(B_n(t)) = t$.

A direct computation shows that $\Psi_N(t,x)=\P_N\Psi(t,x)$
 is a mean-zero real-valued Gaussian random variable with variance
\begin{align*}
 \E \big[\Psi_N(t,x)^2\big] = \E\big[\big(\P_N\Psi(t, x)\big)^2]
 = \s_N
\end{align*}

\noi
for any $t\ge 0$, $x\in\T^2$ and $N \ge 1$,
where $\s_N$ is as in \eqref{sN}.

Let  $u_N$ be as in Theorem \ref{THM:GWP}, 
 satisfying \eqref{RSdSGN} with $\L\big((u_N, \dt u_N)|_{t=0}\big) = \muu_1$.
 Then, write $u_N$ 
 as 
 $u_N = w_N + \Psi$.
 Then, the residual part $w_N$ satisfies the following equation:
 \begin{align}\label{wN}
\begin{cases}
\dt^2 w_N + \dt w_N +(1-\Dl)w_N  + \Im\P_N\big\{e^{i\be \P_Nw_N}\U_N\big\}=0,\\
(w_N,\dt w_N)|_{t = 0}=(0,0).
\end{cases}
\end{align}

\noi
Here, $\U_N$ denotes the so-called 
 imaginary Gaussian multiplicative chaos defined by 
 \begin{align}
 \U_N(t,x)  = \,:\!e^{ i\be\Psi_N(t,x)}\!:\,\,  
 \deff   
\g_N e^{i\be\Psi_N(t,x)}
=  e^{\frac{\be^2}2 \s_N}e^{i\be\Psi_N(t,x)}, 
\label{Ups}
 \end{align} 
 
 \noi
 where $\g_N$ is as in \eqref{CN}.
By proceeding as in 
\cite{HS,ORSW}, 
 we  establish the regularity property of  $\U_N$; 
see Lemma~\ref{LEM:Ups}.
 In particular, given $0 < \be^2 < 4\pi$, 
$\{\U_N\}_{N \in \N}$ 
forms a Cauchy sequence  in $L^p(\O;L^q([0,T];W^{-\al,\infty}(\T^2)))$
 for any finite $p,q\ge 1$ and  $\al>\frac{\be^2}{4\pi}$.
 Then, 
local well-posedness 
of \eqref{wN}, uniformly in $N$, 
follows from a standard contraction argument, 
using 
the Strichartz estimates, 
certain product estimates, and the fractional chain rule.
See Section~\ref{SEC:3}.
The restriction $\be^2 < 2\pi$ appears 
due to a weaker smoothing property in the current wave setting.
See Remark \ref{REM:threshold}.

\begin{remark}\rm
Invariant Gibbs measures for nonlinear wave equations have been studied
extensively, starting with the work \cite{Fri}.
See the survey papers \cite{Oh, BOP4} for the references therein.
In the context of 
the (deterministic) sine-Gordon equation \eqref{SdSG2}, 
McKean \cite{McKean94}
studied the 
 one-dimensional case
and constructed an invariant Gibbs measure for \eqref{SdSG2} on $\T$.
A small adaptation of our argument for proving Theorem \ref{THM:GWP}
allows us to prove
almost sure global well-posedness 
and invariance of the (renormalized) Gibbs measure $\rhoo$
for the (deterministic, renormalized) sine-Gordon equation \eqref{SdSG2} on $\T^2$
for $0 < \be^2 < 2\pi$.

\end{remark}
	
\begin{remark}\label{REM:app}\rm 
In this paper, 
we use a smooth frequency projector 
$\P_N$ with the multiplier $\chi_N$ in \eqref{chi}.
As in the parabolic case, 
it is possible to show that 
the limiting Gibbs measure $\rhoo$ in Theorem~\ref{THM:Gibbs}
and 
the limit $(u, \dt u)$ of $(u_N, \dt u_N)$ in Theorem~\ref{THM:GWP}
are independent of the choice of the smooth cutoff function $\chi$.
See~\cite{OOTz} for such an argument
in the wave case (with a polynomial nonlinearity).
Moreover, 
we may also proceed by smoothing via a mollification
and obtain analogous results.

\end{remark}

\begin{remark}\label{REM:PSG}\rm
As mentioned above, 
the Da Prato-Debussche approach suffices to prove local  well-posedness 
for the parabolic sine-Gordon model \eqref{PSG} in the range $0<\be^2<4\pi$; see the discussion before Theorem~2.1 in \cite{HS}. 
Indeed, with the Da Prato-Debussche decomposition $u_N = w_N + \Psi$, 
where $\Psi$ is the stochastic convolution for the heat case, 
we see that the residual part $w_N$ satisfies 
 \begin{align}\label{wN2}
\dt w_N +\tfrac 12 (1-\Dl)w_N  + \Im\P_N\big\{e^{i\be \P_Nw_N}\U_N\big\}=0.
\end{align}

\noi
Here, 
the  imaginary Gaussian multiplicative chaos $\U_N$
in the heat case has exactly the same regularity as in the wave case stated in Lemma \ref{LEM:Ups}.
Namely, it has the spatial regularity $-\al < - \frac{\be^2}{4\pi}$.
Then, in view of the two degrees of smoothing under the 
heat propagator (the Schauder estimate), 
local well-posedness of \eqref{wN2} for $w_N$ 
in the class
$C([0, T]; W^{2- \al, \infty}(\T^2))$
follows easily from the product estimate (Lemma \ref{LEM:toolbox}\,(iv)),
provided that $\al <  2  -\al$, namely $\be^2 < 4\pi$.

Therefore, combining this local well-posedness, 
the construction of the Gibbs measure (Theorem \ref{THM:Gibbs}), 
and Bourgain's invariant measure argument, 
we conclude almost sure global well-posedness
and invariance of the renormalized Gibbs measure $\rho$ in \eqref{Gibbs4}
for
the parabolic sine-Gordon model \eqref{PSG}.

\end{remark}

\begin{remark}\label{REM:threshold}\rm
In our wave case, 
the linear propagator $\D(t)$ provides
only one degree of smoothing, 
thus requiring $\al - 1 < -\al$ in proving local well-posedness.
This gives the restriction of $\be^2 < 2\pi$ in Theorem~\ref{THM:GWP}.
In view of Theorem \ref{THM:Gibbs}, 
it is therefore of very much interest to study 
further local well-posedness
of the hyperbolic SdSG \eqref{SdSG} for $2\pi \le \be^2 < 4\pi$.

As pointed out in 
\cite{HS, ORW},
the difficulty of the sine-Gordon model on $\T^2$ is heuristically comparable to the one for the dynamical 
$\Phi^3_3$-model when $\be^2=2\pi$.
Namely, when $\be^2 = 2\pi$, 
the hyperbolic SdSG \eqref{SdSG} corresponds to the quadratic SNLW \eqref{SNLW} on $\T^3$
(with $k = 2$).
In \cite{GKO2}, 
Gubinelli, Koch, and the first author proved local well-posedness
of the quadratic SNLW on $\T^3$ by combining
the paracontrolled approach with multilinear harmonic analysis.
Furthermore, in order to replace a commutator argument (which does not provide any smoothing in the dispersive\,/\,hyperbolic setting), they also introduce paracontrolled operators.
Hence, in order to treat the hyperbolic SdSG \eqref{SdSG} for $\be^2 = 2\pi$, 
we plan to adapt the paracontrolled approach as in \cite{GKO2}.

\end{remark}

 \section{Imaginary Gaussian multiplicative chaos
  and construction of the Gibbs measure}
 \label{SEC:2}

In this section, we briefly go over  the regularity and convergence properties of 
the imaginary Gaussian multiplicative chaos $\U_N
= \, :\!e^{i\be\Psi_N}\!:$ defined in \eqref{Ups}. We then proceed to the construction of the Gibbs measure
$\rho$ as stated in Theorem \ref{THM:Gibbs}.
 
 \subsection{Imaginary Gaussian multiplicative chaos}

 In the following, we review 
  the regularity and convergence properties of 
  the truncated stochastic convolution $\Psi_N = \P_N \Psi$,
  where $\Psi$ is as in \eqref{PsiN},  and $\U_N$ in \eqref{Ups}. 
We use the following notation as in \cite{ORSW};  given two functions $f$ and $g$ on $\T^2$, 
we write  
\begin{align*}
f\approx g
\end{align*}

\noi
 if there exist some constants $c_1,c_2\in\R$ such that $f(x)+c_1 \leq g(x) \leq f(x)+c_2$ for any $x\in\T^2\backslash \{0\}
 \cong [-\pi, \pi)^2 \setminus\{0\}$.

We first state the regularity and convergence properties of $\Psi_N$.
See  \cite{GKO, GKO2, ORTz}.

 \begin{lemma}\label{LEM:psi}
Given any $T, \eps>0$ and finite $p\ge 1$, 
$\{(\Psi_N,\dt\Psi_N)\}_{N\in\N}$ is a Cauchy sequence in $L^p(\O;C([0,T];\H^{-\eps}(\T^2)))$, 
thus converging to some limiting process  $(\Psi,\dt\Psi)\in L^p(\O;C([0,T];\H^{-\eps}(\T^2)))$. Moreover, $(\Psi_N,\dt\Psi_N)$ converges almost surely to $(\Psi,\dt\Psi)$ in $C([0,T];\H^{-\eps}(\T^2))$.
 \end{lemma}

Let   $G = (1-\Dl)^{-1}\dl_0$ denote the Green function for $1-\Dl$.
Then, recall from \cite[Lemma 2.3]{ORSW} that for all $N\in\N$ and $x\in\T^2\setminus\{0\}$, we have
\begin{align}\label{G1}
\P_N^2G(x) \approx -\frac1{2\pi}\log\big(|x|+N^{-1}\big).
\end{align}

\noi
Using \eqref{G1}, 
we can proceed as in  the proof of Lemma 2.7 in \cite{ORSW} 
and show that 
 for any $t\ge 0$, the covariance function:
\begin{align*}
\G_N(t,x-y)\deff \E\big[\Psi_N(t,x)\Psi_N(t,y)\big]
\end{align*}
satisfies
\begin{align}\label{covar2}
\G_N(t,x-y)\approx -\frac1{2\pi}\log\big(|x-y|+N^{-1}\big).
\end{align}

\noi
For our problem, 
the stochastic convolution  $\Psi$ defined in \eqref{PsiN}
is a stationary process and thus  $\G_N$ is independent of $t$.
Compare this 
with  the time-dependent case in \cite{ORSW}; see (2.23) in \cite{ORSW}.

Next, we state the regularity and convergence properties of $\U_N$.
%
%
\begin{lemma}\label{LEM:Ups}
Let $0<\be^2<4\pi$. Then,  for any finite $p,q\ge 1$, $T>0$, and  $\al>\frac{\be^2}{4\pi}$, $\{\U_N\}_{N\in\N}$ is a Cauchy sequence in $L^p(\O;L^q([0,T];W^{-\al,\infty}(\T^2)))$
and  hence converges to a limiting process  $\U$ in $L^p(\O;L^q([0,T];W^{-\al,\infty}(\T^2)))$.
\end{lemma}

Due to the stationarity of
$\Psi$, 
 we have $\g_N = e^{\frac{\be^2}{2}\s_N}$ in \eqref{CN} independent of time. 
 This is the reason why, contrary to \cite[Proposition 1.1]{ORSW}, 
 the regularity of $\U$ in Lemma \ref{LEM:Ups} 
is independent of time.
Compare this with Proposition 1.1 in \cite{ORSW}
for the undamped wave case, 
where 
the regularity of the relevant 
imaginary Gaussian multiplicative chaos
decreases over time.

\begin{proof}
 Lemma \ref{LEM:Ups} follows from  a straightforward 
 modification of the proof of  Proposition~1.1 in \cite{ORSW}
on the construction of the imaginary Gaussian multiplicative chaos
in  the undamped wave case.
Namely, using Minkowski's integral inequality, 
it suffices to establish  convergence of $\U_N(t,x)$ for any fixed $t\ge 0$ and $x\in\T^2$,
which  follows from the  argument  in \cite[Proposition 1.1]{ORSW} by replacing \cite[Lemma 2.7]{ORSW} with \eqref{covar2}.  In particular, 
$\frac{\be^2 t}{4\pi}$ in \cite{ORSW} is replaced by 
$\frac{\be^2 }{2\pi}$.
In establishing this lemma (and Proposition 1.1 in \cite{ORSW}), 
we need to exploit a key cancellation property
of charges; 
see Lemma 2.5 in   \cite{ORSW}.
\end{proof}

 \subsection{Construction of the Gibbs measure}

In this subsection, we present a proof of Theorem \ref{THM:Gibbs}.
The main task here is to 
to establish the uniform integrability \eqref{exp1}
of the densities $e^{R_N(u)}$
of  the  weighted Gaussian measures 
$\rhoo_N$ in \eqref{GibbsN}.
For this purpose, 
we use the variational approach
due to 
Barashkov and Gubinelli~\cite{BG}
and
express the partition function $\ZZ_N$ in \eqref{GibbsN}
in terms of a minimization problem
 involving a stochastic control problem
 (Lemma \ref{LEM:vari}). 
 We then study the minimization problem
 and 
establish uniform boundedness of the 
partition function $\ZZ_N$. 
Our argument follows that in Section 4 of \cite{GOTW}.

From \eqref{GibbsN}
and integrating over $\mu_0(v)$, 
we can express the partition function $\ZZ_N$ as 
\begin{equation}
\ZZ_N=  \int e^{ R_N(u)} d\mu_1(u).
\label{Y1}
\end{equation}

\noi
We first show the following convergence property of~$R_N$.
\begin{lemma}\label{LEM:R1}
Given any finite $p \geq 1$, 
$R_N $ defined in \eqref{RN} converges to some limit $R$ in $L^p( \mu_1)$ as $N \to \infty$.

\end{lemma}
\begin{proof}
Let $\L(u) = \mu_1$.
Then, 
from \eqref{RN}, \eqref{CN},  and \eqref{Ups}, we have
\begin{align*}
R_N(u) 
=  \frac{1}{\be}\int_{\T^2}\Re \big(:\!e^{i \be \P_N u}\!: \big)dx 
=  \frac{2\pi}{\be}\Re \ft\U_N(t,n)\big|_{(t, n) = (0, 0)}, 
\end{align*}

\noi
where $\ft \U_N(t, n)$ denotes the spatial Fourier transform at time $t$
and the frequency $n$.
Then, 
Lemma~\ref{LEM:R1} is a direct consequence of the proof of~Lemma \ref{LEM:Ups} (see the proof of Proposition~1.1 in \cite{ORSW}) since the convergence of $\U_N(t,x)$ 
in $L^p(\mu_1)$ is established for any fixed $t\ge 0$ and $x\in\T^2$. 
\end{proof}

Next, we prove the uniform integrability \eqref{exp1}.
Once we prove \eqref{exp1}, 
the desired convergence
\eqref{exp2}
of the density  follows from a standard argument, using 
 Lemma~\ref{LEM:R1} with~\eqref{exp1}. See \cite[Remark 3.8]{Tzvetkov}.
See also the proof of 
Proposition~1.2 in \cite{OTh}.

In order to  prove \eqref{exp1}, 
we follow the argument in \cite{BG,GOTW} 
and derive a variational formula for the partition function
$\ZZ_N$ in \eqref{Y1}. 
Let us first introduce some notations.
See also Section 4 in  \cite{GOTW}. 
Let $W(t)$ be the cylindrical 
 Wiener process  in \eqref{W1}.
 We  define a centered Gaussian process $Y(t)$
by 
\begin{align}
Y(t)
=  \jb{\nabla}^{-1}W(t), 
\label{Y2}
\end{align}

\noi
where $\jb{\nb} = \sqrt{1-\Dl}$.
Then, 
we have $\L(Y(1)) = \mu_1$. 
By setting  $Y_N = \P_NY $, 
we have   $\L(Y_N(1)) = (\P_N)_\#\mu_1$. 
In particular, 
we have  $\E [Y_N(1)^2] = \s_N$,
where $\s_N$ is as in \eqref{sN}.

Next, let $\Ha$ denote the space of drifts, which are the progressively measurable processes that belong to
$L^2([0,1]; L^2(\T^2))$, $\PP$-almost surely. 
Given a drift $\dr \in \Ha$, 
we  define the measure $\Q^\dr$ 
whose Radon-Nikodym derivative with 
respect to $\PP$ is given by the following stochastic exponential:
\begin{align*}
\frac{d\Q^\dr}{d\PP} = e^{\int_0^1 \jb{\dr(t),  dW(t)} - \frac{1}{2} \int_0^1 \| \dr(t) \|_{L^2_x}^2dt},
\end{align*}

\noi
where $\jb{\cdot, \cdot}$ stands for the usual  inner product on $L^2(\T^2)$.
Then, by letting  $\Hc$ denote the subspace of $\Ha$ consisting of drifts such that $\Q^\dr(\O) = 1$,
it follows from Girsanov's theorem (\cite[Theorem 10.14]{DPZ14} and \cite[Theorems 1.4 and 1.7 in Chapter VIII]{RV})
 that $W$ is a semi-martingale under $\Q^\dr$
 with the following  decomposition:
\begin{align}
 W(t) = W^\dr(t) + \int_0^t \dr(t')dt',
\label{Y3}
\end{align}

\noi
 where $W^\dr$ is now a $L^2(\T^2)$-cylindrical Wiener process under
 the new measure $\Q^\dr$.
Substituting \eqref{Y3} in \eqref{Y2} leads to the decomposition:
\begin{align*}
 Y = Y^\dr + \I(\dr), 
\end{align*}

\noi
where 
\begin{align*}
Y^\dr(t) = \jb{\nabla}^{-1} W^\dr(t)\qquad \textup{and}\qquad \I(\dr)(t) = \int_0^t \jb{\nabla}^{-1} \dr(t') dt'.
\end{align*}
In the following, we use $\E$ to denote an expectation 
with respect to $\PP$,
while we use $\E_\Q$ for an expectation with respect to some other probability measure $\Q$.

Proceeding  as  in \cite[Lemma 1]{BG} and \cite[Proposition 4.4]{GOTW}, we then have the following variational formula for the partition function $\ZZ_N$ in \eqref{Y1}.

\begin{lemma} \label{LEM:vari}
For any $N \in \N$, we have
\begin{align} \label{vari}
- \log \ZZ_N = \inf_{\dr \in \Hc} \E_{\Q^\dr} 
\bigg[ - R_N(Y^\dr(1) + \I(\dr)(1)) + \frac{1}{2} \int_0^1 \| \dr(t) \|_{L^2_x}^2 dt \bigg].
\end{align}

\end{lemma}

 Lemma \ref{LEM:vari} follows from a straightforward modification
 of the proof of Proposition 4.4 in  \cite{GOTW}
 and thus we omit details.
  In the following, we 
   use Lemma~\ref{LEM:vari} and show that the infimum in~\eqref{vari}
is bounded away from $- \infty$, uniformly in $N \in \N$,
which 
establishes the uniform bound \eqref{exp1}.
To this end, we first state the following lemma to estimate $Y^\dr(1)$ and $\I(\dr)(1)$.

\begin{lemma}  \label{LEM:Y}
Let $Y^\dr(1)$ and $\I(\dr)(1)$ be as above.

\smallskip

\noi
\textup{(i)} Let $0<\be^2<4\pi$ and $\al>\frac{\be^2}{4\pi}$.
 Then,  for any finite $p \ge 1$, we have 
\begin{align}
\sup_{\dr \in \Hc} \E_{\Q^\dr} 
\| \,:\!e^{ i\be Y^\dr(1)}\!:\, \|_{W^{-\al,\infty}_x}^p  <\infty.
\label{Y7}
\end{align}

\smallskip

\noi
\textup{(ii)} For any $\dr \in \Hc$, we have
\begin{align}
	\| \I(\dr)(1) \|_{H^{1}_x}^2 \leq \int_0^1 \| \dr(t) \|_{L^2_x}^2dt.
\label{Y8}
	\end{align}
\end{lemma}

\begin{proof}
(i) For any $\dr \in \Hc$,
 $W^\dr$ is a cylindrical Wiener process in $L^2(\T^2)$ under $\Q^\dr$.
 Thus,  the law of $Y^\dr(1) = \jb{\nabla}^{-1} W^\dr(1)$ under $\Q^\dr$ is always given by  $\mu_1$, so in particular,  it is independent of $\dr\in\Hc$. 
 Then,~\eqref{Y7} follows
from (the  proof of) Lemma \ref{LEM:Ups}. 

\smallskip
\noi
(ii) The proof of \eqref{Y8} is straightforward from Minkowski's and Cauchy-Schwarz's inequalities.
See the proof of Lemma 4.7 in \cite{GOTW}.
\end{proof}

We are now ready to establish the uniform integrability estimate \eqref{exp1}.
For simplicity, we only prove \eqref{exp1} for $p = 1$.
In view of Lemma \ref{LEM:vari}, we need to bound from below
\begin{align}
\W_N(\dr) = \E_{\Q^\dr} 
\bigg[ - R_N\big(Y^\dr(1) + \I(\dr)(1)\big) + \frac{1}{2} \int_0^1 \| \dr(t) \|_{L^2_x}^2 dt \bigg],	
\label{v_N0}
\end{align}
uniformly in the drift $\dr\in\Hc$ and in $N\in\N$.
To simplify notations, we fix $\dr\in\Hc$ and $N\in\N$ and drop the dependence in $\dr$ and $N$ 
in~\eqref{v_N0}. Moreover, we set $Y = Y^\dr(1) $ and $\Dr = \I(\dr)(1)$.
 
From the definition of $R_N$ in \eqref{RN}, 
we have
\begin{align*}
R_N(Y + \Dr)  & = \frac{1}{\be}\int_{\T^2}\,:\!\cos(\be(Y+\Dr))\!:\,dx\notag\\
& = \frac{1}{\be}
\int_{\T^2}\Big(\,:\!\cos(\be Y)\!: \cos(\be \Dr) \,
- :\!\sin(\be Y)\!:\sin(\be \Dr)\Big)dx.
\end{align*}

\noi
By 
duality 
between $H^\al(\T^2)$ and $H^{-\al}(\T^2)$ and 
Cauchy's inequality, we have
\begin{align}
\begin{split}
|R_N(Y+\Dr)|
&  \les \|:\!\cos(\be Y)\!:\|_{H^{-\al}}
\|\cos(\be \Dr)\big\|_{H^{\al}} + \|:\!\sin(\be Y)\!:\|_{H^{-\al}}
\|\cos(\be \Dr)\|_{H^{\al}}\\
& \les \sum_{\kk \in \{-1, 1\}}\Big\{\dl^{-1} \| : e^{i \kk \be Y}:\|_{H^{-\al}}^2 +  \dl \| e^{i \kk \be \Dr}\|_{H^\al}^2\Big\}
\end{split}
\label{Y9}
\end{align}

\noi
for any $\dl > 0$.
Using the fractional chain rule (see Lemma \ref{LEM:toolbox}\,(ii) below)
and Lemma~\ref{LEM:Y}\,(ii), we have
\begin{align}
\begin{split}
\|e^{ \pm i \be \Dr }\|_{H^{\al}} 
&\sim \|e^{\pm i \be  \Dr}\|_{L^2}+\big\||\nabla|^{\al}\big(e^{\pm  i \be  \Dr})\big)\big\|_{L^2}\\
&\les 1 + \|\Dr\|_{H^\al} \les 1+ \Big( \int_0^1 \| \dr(t) \|_{L^2_x}^2 dt \Big)^\frac{1}{2}, 
\end{split}
\label{Y10}
\end{align}

\noi
as long as $\al \leq 1.$
   Moreover, in view of Lemma~\ref{LEM:Y}\,(i), 
we have    
 \begin{align}
 \E\Big[\|:\! e^{\pm  i \be  Y}\!:\|_{H^{-\al}}^2\Big]
\les 1, 
\label{Y11}
 \end{align}
 
 \noi
 provided that $0 < \be^2 < 4\pi$ and $\al > \frac{\be^2}{4\pi}$.
Therefore,   from \eqref{v_N0}, \eqref{Y9}, \eqref{Y10}, and \eqref{Y11},  
we obtain 
\begin{align*}
\W_N(\dr) 
&\ge \E\bigg[\frac{1}{2} \int_0^1 \| \dr(t) \|_{L^2_x}^2 dt - 
C_1\dl^{-1} -C_2 \dl\Big(1+\int_0^1 \| \dr(t) \|_{L^2_x}^2 dt\Big)\bigg].
\end{align*}

\noi
By taking $\dl > 0$ sufficiently small, we conclude that there exists finite $ C(\dl)>0$ such that
\begin{align*}
\sup_{N \in \mathbb{N}} \sup_{\dr \in \Hc} \W_N(\dr) 
\geq 
\sup_{N \in \mathbb{N}} \sup_{\dr \in \Hc}
\Big\{ -C(\dl) + \frac{1}{4}\int_0^1\|\dr(t)\|_{L^2}^2dt\Big\}
 \geq - C(\dl)
>-\infty.
\end{align*}

\noi
This proves \eqref{exp1}
when $p = 1$.
 The  general case $p\ge 1$ follows from a straightforward modification.

 \section{Local well-posedness of the hyperbolic SdSG}
 \label{SEC:3}
 
In this last section, we present a proof of Theorem \ref{THM:GWP}. 
As mentioned in the introduction, 
thanks to Bourgain's invariant measure argument
and the uniform (in $N$) equivalence of
the (truncated) Gibbs measures and
the base Gaussian measure $\muu_1$, 
it suffices to prove local well-posedness and convergence of 
the truncated dynamics \eqref{RSdSGN}
with the Gaussian random initial data whose law is given by $\muu_1$.
Furthermore, 
in view of 
 the uniform (in $N$)  boundedness
of 
the frequency projector $\P_N$ on $W^{s,p}(\T^2)$, $s\in\R$, $1\le p\le \infty$, 
and the Da Prato-Debussche decomposition:
\begin{align*}
u_N =  w_N + \Psi, 
\end{align*}

\noi
it suffices to prove  local well-posedness 
of the following model equation:
\begin{align}\label{Z0}
\begin{cases}
\dt^2w + \dt w + (1-\Dl)w +  \Im\big\{e^{i\be w}\U\big\} = 0,\\
(w,\dt w)|_{t=0}=(0,0),
\end{cases}
\end{align}

\noi
for a given (deterministic) source function  $\U$.

\begin{proposition}\label{PROP:LWP}
Given $0<\al <\frac12$, 
let $\U$ be a distribution in $L^2([0,1];W^{-\al,\infty}(\T^2))$. 
 Then,  there exists $T = T\big(\|\U\|_{L^2([0,1];W^{-\al,\infty}_x)}\big)\in (0,1]$ 
 and a unique solution $w$ to \eqref{Z0} in the class:
 \begin{align*}
 X^{1-\al}(T)\deff C([0,T];H^{1-\al}(\T^2))\cap C^1([0,T];H^{-\al}(\T^2))\cap L^{\infty}([0,T];L^{\frac2\al}(\T^2)).
 \end{align*} 
 Moreover, the solution map: $\U\mapsto w$ is continuous.
\end{proposition}
\noi

Once we prove Proposition \ref{PROP:LWP}, 
the convergence of the solution $u_N = w_N + \Psi$
to \eqref{RSdSGN}
follows from Lemma \ref{LEM:Ups}
and 
arguing as in our previous work \cite{ORSW}.
Note that the restriction $0<\al<\frac12$ in Proposition~\ref{PROP:LWP} 
gives the range $0<\be^2<2\pi$ in Theorem \ref{THM:GWP} in view of Lemma~\ref{LEM:Ups}.

Before proceeding to the proof of 
Proposition \ref{PROP:LWP}, we state 
the following  deterministic tools from \cite{ORSW}. 
\begin{lemma}\label{LEM:toolbox}
Let $0 < \al < 1$ and $ d \geq 1$. Then, the following estimates hold:

\smallskip
\noi
\textup{(i) (Strichartz estimate).} 
Let $u$ be a solution to  the linear damped wave equation on $\R_+ \times \T^2$:
\begin{align*}
\begin{cases}
\dt^2 u +\dt u + (1-\Dl) u  = f\\
(u,\dt u)|_{t=0}=(u_0,u_1).
\end{cases}
\end{align*}

\noi
Then, for any $0 < T \leq 1$, we have
\begin{align*}
\|u\|_{C_TH^{1-\al}_x} + \|\dt u\|_{C_TH^{-\al}_x} + \|u\|_{L^{\infty}_TL^{\frac2\al}_x} \les \|(u_0,u_1)\|_{\H^{1-\al}} + \|f\|_{L^1_TH^{-\al}_x}.
\end{align*}

\smallskip
\noi
\textup{(ii) (fractional chain rule).} Let $F$ be a Lipschitz function on $\R$ such that $\|F'\|_{L^{\infty}(\R)}\le L$.
Then, for any    $1<p<\infty$, we have 
\begin{align*}
\big\||\nabla|^\al F(f)\big\|_{L^p(\T^d)}\les L\big\||\nabla|^\al f\big\|_{L^p(\T^d)}.
\end{align*}

\smallskip
\noi
\textup{(iii) (fractional Leibniz rule).} Let $1<p_j,q_j,r<\infty$ with $\frac1{p_j}+\frac1{q_j}=\frac1r$, $j=1,2$. Then, we have 
\begin{align*}
\big\|\jb{\nabla}^\al(fg)\big\|_{L^r(\T^d)} 
\les \big\|\jb{\nabla}^\al f\big\|_{L^{p_1}(\T^d)}\|g\|_{L^{q_1}(\T^d)} + \|f\|_{L^{p_2}(\T^d)}\big\|\jb{\nabla}^\al g\big\|_{L^{q_2}(\T^d)}.
\end{align*}

\smallskip
\noi
\textup{(iv) (product estimate).} 
Let $1<p,q,r<\infty$ such that $\frac1p+\frac1q\le \frac1r + \frac{\al}d$. Then, we have
\begin{align*}
\big\|\jb{\nabla}^{-\al}(fg)\big\|_{L^r(\T^d)} \les \big\|\jb{\nabla}^{-\al}f\big\|_{L^p(\T^d)}
\big\|\jb{\nabla}^\al g\big\|_{L^r(\T^d)}.
\end{align*}
\end{lemma}

The Strichartz estimate on $\T^2$ in (i) follows from 
the corresponding Strichartz estimate for the wave/Klein-Gordon equation
on $\R^2$ 
(see \cite{GV,KeelTao,KSV})
for the 
 $(1- \al)$-wave admissible pair  $(\infty,\frac2\al)$, 
the finite speed of propagation, 
and the fact that the linear damped wave propagator $\D(t)$
in \eqref{D}
satisfies the same Strichartz estimates
as that for 
 the Klein-Gordon equation $\dt^2u +(\frac34-\Dl)u =0$. 
For the fractional chain rule on $\T^d$, \cite{Gatto}.
See \cite{GKO} for (iii) and (iv).

We now present a proof of Proposition \ref{PROP:LWP}.

\begin{proof}[Proof of Proposition \ref{PROP:LWP}.]
By writing \eqref{Z0} in the Duhamel formulation, we have
\begin{align*}
w(t) = \Phi(w)(t): =  -\int_0^t\D(t-t')\Im\big\{e^{i\be w}\U\big\}(t')dt',
\end{align*}

\noi
where $\D(t)$ is as in \eqref{D}.
Fix  $0 < T \leq 1$
and $0 < \al < \frac 12$.
We use $B$ to denote the ball in $X^{1-\al}(T)$ of radius $1$
centered at the origin.

From  Lemma \ref{LEM:toolbox} (i), (iv), and then (ii)
with $\al < 1 - \al$, 
we have
\begin{align}
\begin{split}
\| \Phi(w) \|_ {X^{1-\al}(T)} 
&\les 
\|  e^{i\be w}\U \|_{L^1_TH^{-\al}_x}
\les  T^{\frac12}\|e^{i\be w}\|_{L^{\infty}_TH^{\al}_x}\|\U\|_{L^2_TW^{-\al,\frac2\al}_x}\\
 &\les T^{\frac12} \big( 1 + \|w \|_{X^{1-\al}(T)}\big)\|\U\|_{L^2_TW^{-\al,\infty}_x}\\
  &\les T^{\frac12} \|\U\|_{L^2_TW^{-\al,\infty}_x}
\end{split}
\label{Z1}
\end{align}

\noi
for $w \in B$.
By the fundamental theorem of calculus, we have
\begin{align*}
e^{i\be w_1}-e^{i\be w_2} = (w_1-w_2)F(w_1,w_2)\deff (w_1-w_2)(i\be)\int_0^1e^{i\be(\tau w_1 + (1-\tau)w_2)}d\tau.
\end{align*}

\noi
Thus, 
from  Lemma \ref{LEM:toolbox} (i) and (iv), we have
\begin{align}
\| \Phi(w_1)-\Phi(w_2) \|_ {X^{1-\al}(T)} 
&\les  T^{\frac12}\|(w_1-w_2)F(w_1,w_2)\|_{L^{\infty}_TW^{\al,\frac2{1+\al-\eps}}_x}\|\U\|_{L^2_TW^{-\al,\frac2\eps}_x}
\label{Z2}
\end{align}

\noi
for any small $\eps >0$.
Then, by applying  Lemma \ref{LEM:toolbox} (iii) and then (ii) to \eqref{Z2}, 
we obtain
\begin{align}
\begin{split}
\|(  w_1 &  -w_2)F(w_1,w_2)\|_{L^{\infty}_TW^{\al,\frac2{1+\al-\eps}}_x}\\
&\les \|w_1-w_2\|_{L^{\infty}_TH^{\al}_x}\|F(w_1,w_2)\|_{L^{\infty}_TL^{\frac2{\al-\eps}}_x} \\
& \hphantom{X}
+\|w_1-w_2\|_{L^{\infty}_TL^{\frac2\al}_x}\|F(w_1,w_2)\|_{L^{\infty}_TW^{\al,\frac2{1-\eps}}_x}\\
&\les \|w_1-w_2\|_{X^{1-\al}(T)}\Big(1+\|w_1\|_{L^{\infty}_TW^{\al,\frac2{1-\eps}}_x}+\|w_2\|_{L^{\infty}_TW^{\al,\frac2{1-\eps}}_x}\Big).
\end{split}
\label{Z3}
\end{align}

\noi
Given $0 < \al < \frac 12$, choose $\eps > 0$ small such that 
$\al + \eps < 1 - \al$.
Then, it follows from 
\eqref{Z2},  \eqref{Z3}, and Sobolev's inequality that 
\begin{align}
\| \Phi(w_1)-\Phi(w_2) \|_ {X^{1-\al}(T)} 
&\les T^{\frac12}
\|\U\|_{L^2_TW^{-\al,\infty}_x}
\|w_1-w_2\|_{X^{1-\al}(T)}
\label{Z4}
\end{align}

\noi
for any  $w_1,w_2\in B$.

Hence, we conclude from 
 \eqref{Z1} and \eqref{Z4}
 that the map $\Phi = \Phi_{\U}$ is a contraction on $B\subset X^{1-\al}(T)$,
  provided that $T=T\big(\|\U\|_{L^2([0, 1]; W^{-\al,\infty}_x}\big)>0$ is sufficiently small.
  The uniqueness in the whole space $X^{1-\al}(T)$ follows from a standard continuity argument, 
  while a small modification of
  the argument above shows  the continuous dependence on  $\U$.
\end{proof}

Proposition \ref{PROP:LWP} thus establishes 
local well-posedness of the truncated equation~\eqref{wN}, 
uniformly in $N\in\N$,  
and also for the limiting equation
\begin{align}\label{w}
\begin{cases}
\dt^2w + (1-\Dl)w + \dt w + \Im\big\{e^{i\be w}\U\big\} = 0,\\
(w,\dt w)|_{t=0}=(0,0), 
\end{cases}
\end{align}

\noi
where $\U$ is the limit of $\U_N$ constructed in Lemma \ref{LEM:Ups}.
We briefly describe an extra ingredient in 
showing 
 convergence of $w_N$  to  $w$, satisfying  \eqref{w}. 
 Since the flow map constructed in Proposition \ref{PROP:LWP} is continuous in $\U$, 
 there is only one extra term 
 $(\Id - \P_N)\big\{e^{i\be w}\U\big\}$
 in estimating  the difference $\|w_N-w\|_{C_TH^{1-\al}_x}$.
By exploiting the fact that this extra term is supported on 
high frequencies  $\{|n|\ges N\}$, we have 
\begin{align*}
\big\|(\Id-\P_N)\big\{e^{i\be w}\U\big\}\big\|_{L^1_TH^{-\al}_x} 
&\les N^{-\eps}\big\|e^{i\be w}\U\big\|_{L^1_TH^{-\al+\eps}_x}\\
&\les T^{\frac12}N^{-\eps}\big(1+\|w\|_{L^{\infty}_TH^{\al-\eps}_x}\big)\|\U\|_{L^2_TW^{-\al+\eps,\infty}_x}.
\end{align*}

\noi
Combining with the argument above, 
we can then prove  convergence  $w_N\to w$ as $N\to\infty$. 
Note that given $0<\be^2<2\pi$ and $0<\al<\frac12$ with $\frac{\be^2}{4\pi}<\al$, 
we have  $\frac{\be^2}{4\pi}<\al-\eps$ for small $\eps > 0$, 
which guarantees that  $\U\in L^2([0,T];W^{\eps-\al}(\T^2))$ in view of Lemma \ref{LEM:Ups}.

\begin{acknowledgment}

\rm 
T.O.~and T.R.~were supported by the European Research Council (grant no.~637995 ``ProbDynDispEq'').
P.S.~was partially supported by NSF grant DMS-1811093.

\end{acknowledgment}

\end{document}